\documentclass[1p]{elsarticle}

\usepackage{amsmath}
\usepackage{amsthm}
\usepackage{amssymb}
\usepackage{graphicx}
\usepackage{enumitem}
\usepackage{environ}
\usepackage[e]{esvect}
\usepackage{mathtools}
\RequirePackage[]{algorithm2e}
\usepackage{calc}
\usepackage{tikz,pgfplots}
\usepackage{todonotes}
\usepackage{multicol}
\usepackage[utf8x]{inputenc}
\usepackage{hyperref}

\usepackage{listings}
\newtheorem{theorem}{Theorem}[section]

\newtheorem{corollary}[theorem]{Corollary}

\theoremstyle{definition}

\newtheorem{problem}[theorem]{Problem}

\theoremstyle{remark}

\theoremstyle{conjecture}
\newtheorem{conjecture}[theorem]{Conjecture}

\usepackage[sc]{mathpazo}
\usepackage[T1]{fontenc}

\numberwithin{equation}{section}

\newcommand{\comment}[1]{}

\begin{document}

\begin{frontmatter}
\title{A lower--bound for the number of conjugacy classes of $A_n$}

\author{Xandru Mifsud\corref{mycorrespondingauthor}}
\address{Department of Mathematics, University of Malta, Msida, Malta}
\ead{xmif0001@um.edu.mt}

\begin{abstract}
    We establish a sharp lower--bound for the number of conjugacy classes $k(A_n)$ in the alternating group $A_n$, for $n \geq 3$. Namely, we show that $k\left(A_n\right) \geq \frac{k\left(A_7\right)}{\log_2\left|A_7\right|} \cdot \log_2\left|A_n\right|$ with equality if, and only if, $n = 7$. The observations leading towards this result were obtained through a \textit{genetic algorithm} developed to search for groups having certain properties.
\end{abstract}

\end{frontmatter}

\section{Introduction}

One of the problems in group theory is the classification of finite groups by the number of conjugacy classes. Given a group $G$, we denote its number of conjugacy classes by $k(G)$ and its order by $|G|$. Unless otherwise stated, all groups considered throughout this work are finite. All groups with $k(G) \leq 14$ have been classified, with the cases for $k(G) = 13$ and $k(G) = 14$ being resolved recently in \cite{Lopez}. Consequently, bounding $k(G)$ is an important problem in group theory. 

A number of bounds for $k(G)$ have been established in the literature. Recently, Maróti \cite{MarotiKBound} showed that if $G$ is a finite group whose order is divisible by a prime $p$, then $k(G) \geq 2\sqrt{p - 1}$. Of notable interest are bounds establishing explicit relations between $k(G)$ and $|G|$. Erdős and Turán had established that $k(G) \geq \log_2 \log_2 |G|$. One of the main problems in this area is resolving the following conjecture on the relation between $k(G)$ and $\log_2 |G|$. 

\begin{conjecture}\label{Conj1}
There exists a constant $C > 0$ such that for any finite group $G$, $k(G) \geq C \log_2 |G|$.
\end{conjecture}

One of the sharpest bounds obtained in this regard is the Pyber--Keller bound, first established by Pyber \cite{Pyber} and then further sharpened by Keller \cite{Keller}.
\begin{theorem}[\cite{Pyber, Keller}]
There exists a (explicitly computable) constant $C > 0$ such that every finite group $G$ with $|G| \geq 4$ satisfies $$k(G) \geq \dfrac{C}{\left(\log_2 \log_2 |G|\right)^7} \cdot \log_2 |G|.$$ 
\end{theorem}

In light of Conjecture \ref{Conj1}, we prove the following sharp--bound of this type, for alternating groups.

\begin{theorem} \label{Thm1}
    Let $n \geq 3$. Then, $$k\left(A_n\right) \geq \dfrac{k\left(A_7\right)}{\log_2\left|A_7\right|} \cdot \log_2\left|A_n\right|$$
	with equality if, and only if, $n = 7$.
\end{theorem}

\subsection{Overview}

We begin by introducing some preliminaries in Section 2, primarily some textbook results on the relations between the number of conjugacy classes of $S_n$ and the number of partitions of a positive integer $n$, denoted by $p(n)$. This is followed by a proof of Theorem \ref{Thm1} in Section 3, and a short discussion in Section 4 on the result and some open problems.

\section{Preliminaries: Relations between $k(S_n)$, $k(A_n)$ and $p(n)$}

A classical result in the study of $S_n$ is that the conjugacy classes of $S_n$ correspond to the permutation cycle structures, and in turn to the partitions of $n$. Hence $k(S_n) = p(n)$. While a generating--function for the partition function $p(n)$ is known, no closed--form expression is known. Consequently, the study of (asymptotic) bounds and recurrence relations for $p(n)$ represents a highly active area of research in number theory. Maróti \cite{MarotiPBound} established a number of lower bounds for $p(n)$, namely the following, which will be useful in our proof for Theorem \ref{Thm1}.

\begin{theorem}[\cite{MarotiPBound}]\label{PBound}
For every positive integer $n$, $$\dfrac{e^{2\sqrt{n}}}{14} < p(n).$$
\end{theorem}

Recall that $A_n$ is the group of even permutations in $S_n$. The \textit{splitting criterion} for conjugacy classes in $A_n$ states that given a conjugacy class of even permutations in $S_n$: the conjugacy class is either a conjugacy class in $A_n$ or else it `splits' into two conjugacy classes in $A_n$, with splitting occurring if and only if the cycle structure corresponding to the conjugacy class consists of cycles having distinct odd lengths. Consequently, the number of conjugacy classes in $S_n$ that split in $A_n$ is equal to the number of partitions of $n$ into distinct odd parts, which we term as $\alpha(n)$. 

Let $\beta(n)$ be the number of conjugacy classes of even permutations in $S_n$ that do not split in $A_n$.  A classical result relating $k(A_n)$ to $k(S_n)$, via $\alpha(n)$ and $\beta(n)$, is the following.

\begin{theorem}\label{AlphaBetaThm}
Let $n \geq 3$. Then, $k(S_n) = \alpha(n) + 2\beta(n)$ and $k(A_n) = 2\alpha(n) + \beta(n)$.
\end{theorem}

Consequently, we get the following corollary.

\begin{corollary}\label{CorrAlphaBeta}
Let $n \geq 3$. Then, $\dfrac{k(S_n)}{2} \leq k(A_n)$.
\end{corollary}
\begin{proof}
By Theorem \ref{AlphaBetaThm}, we have $\dfrac{k(S_n)}{2} = \dfrac{\alpha(n)}{2} + \beta(n) \leq 2\alpha(n) + \beta(n) = k(A_n)$, since $\alpha(n) \geq 0$.
\end{proof}

Further aspects of these results are given by Dénes, Erdős and Turán \cite{DET}.

\section{Proof of Theorem 1.3}

We are now in a position to prove our main result. We will show that
\begin{equation}
\dfrac{k(A_7)}{\log_2 |A_7|} = \min_{n \geq 3} \dfrac{k(A_n)}{\log_2 |A_n|}.
\end{equation}

For $3 \leq n \leq 11$, by explicit computation\footnote{Using GAP \cite{GAP}: \lstinline{gap> for n in [3..11] do Print(n, " : ", NrConjugacyClasses(AlternatingGroup(n)) / Log2(Float(Order(AlternatingGroup(n)))), "\n"); od;}} (see Table 1), we have that $\frac{k(A_7)}{\log_2 |A_7|} \leq \frac{k(A_n)}{\log_2 |A_n|}$.

\renewcommand{\arraystretch}{1.7}
\begin{table}[h]
\center
\begin{tabular}{|c|l||c|l||c|l|}
\hline $n$ & $\frac{k(A_n)}{\log_2 |A_n|}$ & $n$ & $\frac{k(A_n)}{\log_2 |A_n|}$ & $n$ & $\frac{k(A_n)}{\log_2 |A_n|}$ \\ \hline \hline
3   & $1.892 \dots$                  & 6   & $0.824 \dots$                                                             & 9   & $1.030 \dots$                                                                \\ \hline
4   & $1.115 \dots$                  & 7   & $0.796 \dots$                                                             & 10  & $1.154 \dots$                                                                \\ \hline
5   & $0.846 \dots$                  & 8   & $0.979 \dots$                                                             & 11  & $1.278 \dots$                                                                \\ \hline
\end{tabular}
\caption{Values of $\frac{k(A_n)}{\log_2 |A_n|}$ for $3 \leq n \leq 11$, computed using GAP \cite{GAP}.}
\end{table}

We prove the case when $n \geq 12$. Firstly note that $\log_2 |A_n| \leq (n - 2)\log_2 n$ since, 
$$\log_2 \frac{n!}{2} = -1 + \sum\limits_{r = 1}^{n} \log_2 r = \sum\limits_{r = 3}^{n} \log_2 r \leq (n - 2)\log_2 n$$
noting that $\log_2 (x)$ is monotonically increasing, $\log_2 1 = 0$ and $\log_2 2 = 1$.

Using this inequality and Corollary \ref{CorrAlphaBeta}, we obtain
$$\dfrac{k(A_n)}{\log_2 |A_n|} \geq \dfrac{k(S_n)}{2(n-2)\log_2 n}$$
and by Theorem \ref{PBound} and the fact that $p(n) = k(S_n)$, we have that:

\begin{equation}
\dfrac{k(A_n)}{\log_2 |A_n|} > t(n), \ \ \ t(n) = \dfrac{e^{2\sqrt{n}}}{28(n-2)\log_2 n}.
\end{equation}

Since, for large enough $n$, $28(n-2)\log_2 n << e^{2\sqrt{n}}$, then in particular there exists a positive integer $N$ such that for all $n \geq N$, $1 \leq t(N) \leq t(n)$. In particular for all $n \geq 12$, $$1 < 1.016 \ldots = t(12) \leq t(n)$$
and consequently by (3.2) and the case for $A_7$ in Table 1,

$$\dfrac{k(A_n)}{\log_2 |A_n|} > 1 > \dfrac{k(A_7)}{\log_2 |A_7|}$$
and therefore (3.1) follows. 

By the strict inequality in (3.2) and Table 1, it follows that equality holds if, and only if, $n = 7$. \qed

\section{Concluding remarks}

\begin{enumerate}[label=\textbf{(\arabic*)}]
\item This lower--bound was initially suggested through empirical analysis, using a genetic algorithm developed using the Python API of SageMath \cite{Sage}. Davies \textit{et al.} \cite{AIMath} recently highlighted the effectiveness of guiding mathematical intuition in this manner, through the use of artificial intelligence.

\item Since the number of characters of an Abelian group $G$ is $|G|$ and $\frac{k(A_7)}{\log_2 |A_7|} < 1$, then when $G$ is non--trivial we have that $\frac{k(G)}{\log_2 |G|} \geq 1 > \frac{k(A_7)}{\log_2 |A_7|}$.

\item Based on the result of Theorem \ref{Thm1} and \textbf{(2)} above, as well as further empirical analysis carried out using our algorithm, we hypothesise that the constant $C$ in Conjecture \ref{Conj1} is `close' to $\frac{k(A_7)}{\log_2 |A_7|}$, such that it holds with high probability. Hence we pose the following problem.
	\begin{problem}\label{Conj2}
	Which non--Abelian finite groups $G$ satisfy $\frac{k(G)}{\log_2 |G|} < \frac{k(A_7)}{\log_2 |A_7|}$?
	\end{problem}

Case--by--case analysis yields that $M_{11}$ (the Mathieu group of degree 11) is the smallest such group satisfying Problem \ref{Conj2}.\\

If $G$ is a group as in Problem \ref{Conj2} which embeds into $A_n$ for some $n \in \mathbb{N}$, then $k(G)$ is a lower--bound for $k(A_n)$, as follows.

\begin{corollary}
Let $G$ be a finite non--Abelian subgroup of $A_n$ for some $n \in \mathbb{N}$. If $\frac{k(G)}{\log_2 |G|} < \frac{k(A_7)}{\log_2 |A_7|}$, then $k(G) < k(A_n)$. 
\end{corollary}

\begin{proof}
Follows immediately from the fact that $|G| \leq |A_n|$ and (3.1), namely since: $$k(G) < \log_2 |G| \cdot \dfrac{k(A_7)}{\log_2 |A_7|} \leq \log_2 |G| \cdot \dfrac{k(A_n)}{\log_2 |A_n|} < k(A_n).$$
\end{proof}
\end{enumerate}

\section*{Acknowledgements}

The author would like to thank Prof. Josef Lauri for introducing the author to Conjecture \ref{Conj1} and Adriana Baldacchino for their helpful suggestions in the development of the genetic algorithm used to deduce the bound.

\bibliography{biblio.bib}

\newpage

\appendix
\section{Searching for groups through genetic algorithms}

We briefly outline in this section the rationale behind our algorithmic work which helped in the deduction of the main result of this short--paper. First and foremost, we designed a \textit{score function} associated with the parameters which we are interested in. In particular, let $\mathcal{G}$ be a collection of finite permutation groups. We defined a score--function $s \colon \mathcal{G} \to \mathbb{R}^+$ which maximises $\frac{\log_2 |G|}{k(G)}$ for $G \in \mathcal{G}$, \textit{ie.} the highest scoring groups are those which minimise $\frac{k(G)}{\log_2 |G|}$, since we are after Conjecture \ref{Conj1}.

We then generate a random population of permutation groups, each one constructed from a random set of permutations and obtaining the permutation group they generate using the \verb|PermutationGroup| functionality of \verb|SageMath|. Optionally, one could also enrich the population with `known' groups such as cyclic groups, alternating groups, etc. 

This population is then \textit{evolved} through our genetic algorithm over a number of pre--determined iterations, keeping the highest scoring groups from one iteration to the next. 

To evolve the population, besides scoring and selecting the groups in the population, we include mechanisms to introduce new groups from the highest scoring ones into the mix, by a number of \textit{mutation} and \textit{cross--over} operations.

\subsection{Cross--over operations}

We cross--over between two groups in one of two methods:
\begin{enumerate}
	\item[i.] either by introducing the direct--product of the two permutation groups into the population,
	\item[ii.] or by considering a generating set of each of the two permutation groups, randomly choosing a subset of permutations from each, and adding to the population the group generated by the union of these resulting subsets. 
\end{enumerate}

In this manner we aim at being able to search for groups of larger order than in the current population (through the direct--product) and which inherit traits from two different high--scoring groups (by mixing their generators).

\subsection{Mutation operations}

The mutation operations considered here take a group and return a new one, by mutating the generating set of the existing group and then returning whatever group is generated by the mutated generating set. We outline here a number of mutations which we explored:
\begin{enumerate}
\item[i.] a mutation could be the random removal of a generator from the existing generating set,
\item[ii.] or the addition of a new randomly generated permutation to the existing generating set,
\item[iii.] or the addition of a new generator of the form $ghg^{-1}h^n$ for some existing generators $g$ and $h$,
\item[iv.] or the addition of a new generator of the form $(gh)^n$ for some existing generators $g$ and $h$.
\end{enumerate}

In this manner we aim at enriching the population which both `random' as well as well--understood structure.

\subsection{Conclusion}

In this manner, starting from different randomly generated populations, we consistently converged towards the group $A_7$ as being the highest scoring one, with other alternating groups scoring similarly (but slightly below), thereby suggesting the result proved in this paper once the link was established.

Further searches with a larger population and increased iterations eventually yielded the Mathieu group of degree 11, $M_{11}$, as a higher scoring group than $A_7$, as noted earlier.

\end{document}